\documentclass[12pt]{amsart}
\date{\today}

\usepackage[cp1251]{inputenc}         

\usepackage[ukrainian]{babel} 
\usepackage{amsmath,amsthm,amsfonts,amssymb,mathrsfs}

\usepackage{color}

\usepackage{hyperref}

\newtheorem{theorem}{Теорема}[section]

\newtheorem{proposition}[theorem]{Твердження}
\newtheorem{corollary}[theorem]{Наслiдок}
\newtheorem{lemma}[theorem]{Лема}
\theoremstyle{definition}
\newtheorem{construction}[theorem]{Конструкція}
\newtheorem{remark}[theorem]{Зауваження}

\begin{document}

\title[Напiвгрупа монотонних ко-скiнченних часткових гомеоморфiзмiв дiйсно{\"{\i}} прямо{\"{\i}}]{Напiвгрупа монотонних ко-скiнченних часткових гомеоморфiзмiв дiйсно{\"{\i}} прямо{\"{\i}}}

\author[Олег~Гутік, Катерина Мельник]{Олег~Гутік, Катерина Мельник}
\address{Механіко-математичний факультет, Львівський національний університет ім. Івана Франка, Університецька 1, Львів, 79000, Україна}
\email{o\_gutik@franko.lviv.ua, ovgutik@yahoo.com, chepil.kate@gmail.com}

\keywords{Semigroup of partial homeomorphisms, group of homeomorphisms, inverse semigroup, minimum group congruence, semidirect product, real line.}

\subjclass[2010]{20M15,  20M50, 18B40}

\begin{abstract}
У даній праці досліджується структура напівгрупи $\mathscr{P\!\!H}^+_{\!\!\operatorname{\textsf{cf}}}\!(\mathbb{R})$ усіх монотонних ко-скін\-чен\-них часткових гомеоморфізмів звичайної дійсної прямої $\mathbb{R}$. Доведено, що інверсна напівгрупа $\mathscr{P\!\!H}^+_{\!\!\operatorname{\textsf{cf}}}\!(\mathbb{R})$ є факторизовною та $F$-інверсною. Описано структуру в'язки напівгрупи $\mathscr{P\!\!H}^+_{\!\!\operatorname{\textsf{cf}}}\!(\mathbb{R})$, її двобічні ідеали, максимальні підгрупи та відношення Гріна на ній. Доведено, що фактор-напівгрупа $\mathscr{P\!\!H}^+_{\!\!\operatorname{\textsf{cf}}}\!(\mathbb{R})/\mathfrak{C}_{\textsf{mg}}$ за найменшою груповою конгруенцією $\mathfrak{C}_{\textsf{mg}}$ ізоморфна групі $\mathscr{H}^+\!(\mathbb{R})$ усіх гомеоморфізмів, що зберігають орієнтацію простору $\mathbb{R}$, а також, що напівгрупа $\mathscr{P\!\!H}^+_{\!\!\operatorname{\textsf{cf}}}\!(\mathbb{R})$ ізоморфна напівпрямому добутку $\mathscr{H}^+\!(\mathbb{R})\ltimes_\mathfrak{h}\mathscr{P}_{\!\infty}(\mathbb{R})$ вільної напівгратки з одиницею $(\mathscr{P}_{\!\infty}(\mathbb{R}),\cup)$ з групою $\mathscr{H}^+\!(\mathbb{R})$.

\medskip

\textbf{{Oleg Gutik, Kateryna Melnyk},} %
\textbf{{The semigroup of monotone co-finite partial homeomorphisms of the real line}}%
 
{In the paper we investigate the semigroup of monotone co-finite partial homeomorphisms of the space of the usual real line $\mathbb{R}$. We prove that the inverse semigroup $\mathscr{P\!\!H}^+_{\!\!\operatorname{\textsf{cf}}}\!(\mathbb{R})$ is factorizable and $F$-inverse. We describe the structure of the band of the semigroup $\mathscr{P\!\!H}^+_{\!\!\operatorname{\textsf{cf}}}\!(\mathbb{R})$, its two-sided ideals, maximal subgroups and Green's relations. We prove that the quotient semigroup $\mathscr{P\!\!H}^+_{\!\!\operatorname{\textsf{cf}}}\!(\mathbb{R})/\mathfrak{C}_{\textsf{mg}}$, where $\mathfrak{C}_{\textsf{mg}}$ is the maximum group congruence on $\mathscr{P\!\!H}^+_{\!\!\operatorname{\textsf{cf}}}\!(\mathbb{R})/\mathfrak{C}_{\textsf{mg}}$, is isomorphic to the group of all oriental homeomorphisms of the space $\mathbb{R}$, and showe that the semigroup  $\mathscr{P\!\!H}^+_{\!\!\operatorname{\textsf{cf}}}\!(\mathbb{R})$ is isomorphic to a semidirect product $\mathscr{H}^+\!(\mathbb{R})\ltimes_\mathfrak{h}\mathscr{P}_{\!\infty}(\mathbb{R})$ of the free semilattice with unit  $(\mathscr{P}_{\!\infty}(\mathbb{R}),\cup)$ by the group $\mathscr{H}^+\!(\mathbb{R})$.}
\end{abstract}

\maketitle


\section{Термінологія та означення}

В даній праці ми користуватимемося термінологією з \cite{CliffordPreston1961-1967, Engelking1989, Petrich1984}.

Надалі у тексті потужність множини $A$ позначатимемо через $|A|$, перший не\-скін\-чен\-ний кардинал через $\omega$, і множину натуральних чисел --- через $\mathbb{N}$. Також, будемо вважати, що на множині дійсних чисел $\mathbb{R}$ визначена звичайна (евклідова) топологія.

Якщо визначене часткове відображення $\alpha\colon X\rightharpoonup Y$ з множини $X$ у множину $Y$, то через $\operatorname{dom}\alpha$ i $\operatorname{ran}\alpha$ будемо позначати його \emph{область визначення} та \emph{область значень}, відповідно, а через $(x)\alpha$ і $(A)\alpha$ --- образи елемента $x\in\operatorname{dom}\alpha$ та підмножини $A\subseteq\operatorname{dom}\alpha$ при частковому відображенні $\alpha$, відповідно. Часткове відображення $\alpha\colon X\rightharpoonup Y$ називається \emph{ко-скінченним}, якщо множини $X\setminus\operatorname{dom}\alpha$ та $Y\setminus\operatorname{ran}\alpha$ є скінченними.

Часткове відображення $\alpha\colon \mathbb{R}\rightharpoonup\mathbb{R}$ називається \emph{частковим гомеоморфізмом} простору $\mathbb{R}$, якщо його звуження $\alpha|_{\operatorname{dom}\alpha} \colon \operatorname{dom}\alpha\rightarrow \operatorname{ran}\alpha$ є гомеоморфізмом.

Рефлексивне, антисиметричне та транзитивне відношення на множині $X$ називається \emph{частковим порядком} на $X$. Множина $X$ із заданим на ній частковим порядком $\leqslant$ називається \emph{частково впорядкованою множиною} і позначається $(X,\leqslant)$.

Елемент $x$ частково впорядкованої множини $(X,\leqslant)$ називається:
\begin{itemize}
  \item \emph{максимальним} (\emph{мінімальним}) в $(X,\leqslant)$, якщо з відношення $x\leqslant y$ ($y\leqslant x$) в $(X,\leqslant)$ випливає рівність $x=y$;
  \item \emph{найбільшим} (\emph{найменшим}) в $(X,\leqslant)$, якщо $y\leqslant x$ ($x\leqslant y$) для всіх $y\in X$.
\end{itemize}

У випадку, якщо $(X,\leqslant)$~--- частково впорядкована множина і $x\leqslant y$, для деяких $x,y\in X$, то будемо говорити, що елементи $x$ та $y$~--- \emph{порівняльні} в $(X,\leqslant)$. Якщо ж для елементів $x\leqslant y$ не виконується жодне з відношень $x\leqslant y$ або $y\leqslant x$, то говоритимемо, що елементи $x$ та $y$ є непорівняльними в частково впорядкованій множині $(X,\leqslant)$. Частковий порядок $\leqslant$ на $X$ називається \emph{лінійним}, якщо довільні два елементи в $(X,\leqslant)$ є порівняльними.

Підмножина $A$ частково впорядкованої множини $(X,\leqslant)$ називається:
\begin{itemize}
  \item \emph{ланцюгом}, якщо відношення $\leqslant$ індукує на $A$ лінійний порядок;
  \item \emph{антиланцюгом}, якщо довільні два різні елементи в $A$ є непорівняльними стосовно індукованого з $(X,\leqslant)$ часткового порядку.
\end{itemize}
З леми Цорна випливає, що кожен ланцюг (антиланцюг) $A$ частково впорядкованої множини $(X,\leqslant)$ міститься в максимальному ланцюзі (антиланцюзі) $B$, стосовно від\-но\-шен\-ня включення підмножин, а отже кожен ланцюг (антиланцюг) в $(X,\leqslant)$ можна доповнити (необов'язково єдиним чином) до максимального ланцюга (антиланцюга). Лінійно впорядкована множина, яка є порядково ізоморфна множині від'ємних цілих чисел $\{-1, -2, -3, -4, \ldots\}$ зі звичайним порядком $\leqslant$ називається $\omega$-ланцюгом.

Відображення $h\colon X\rightarrow Y$ з частково впорядкованої множини $(X,\leqslant)$ в частково впорядковану множину $(Y,\leqslant)$ називається \emph{монотонним}, якщо з $x\leqslant y$ випливає $(x)h\leqslant (y)h$.

Якщо $S$~--- напівгрупа, то її підмножина ідемпотентів позначається через $E(S)$.  Напівгрупа $S$ називається \emph{інверсною}, якщо для довільного її елемента $x$ існує єдиний елемент $x^{-1}\in S$ такий, що $xx^{-1}x=x$ та $x^{-1}xx^{-1}=x^{-1}$. В інверсній напівгрупі $S$ вище означений елемент $x^{-1}$ називається \emph{інверсним до} $x$. \emph{В'язка}~--- це напівгрупа ідемпотентів, а \emph{напівгратка}~--- це комутативна в'язка. Надалі через $(\mathscr{P}_{\!\infty}(\mathbb{R}),\cup)$ по\-зна\-ча\-ти\-ме\-мо \emph{вільну напівгратку} з одиницею над множиною дійсних чисел, тобто множину усіх скінченних (включно з порожньою) підмножин множини $\mathbb{R}$ з операцією об'єднання.

Відношення еквівалентності $\mathfrak{K}$ на напівгрупі $S$ називається \emph{конгруенцією}, якщо для елементів $a$ i $b$ напівгрупи $S$ з того, що виконується умова $(a,b)\in\mathfrak{K}$ випливає, що $(ca,cb), (ad,bd) \in\mathfrak{K}$, для всіх $c,d\in S$. Відношення $(a,b)\in\mathfrak{K}$ ми також будемо записувати $a\mathfrak{K}b$, і в цьому випадку будемо говорити, що \emph{елементи $a$ i $b$ є $\mathfrak{K}$-еквівалентними}. На кожній напівгрупі $S$ існують наступні конгруенції: \emph{універсальна} $\mathfrak{U}_S=S\times S$ та \emph{одинична} (\emph{діагональ}) $\Delta_S=\{(s,s)\colon s\in S\}$. Такі конгруенції називаються \emph{тривіальними}. Кожен двобічний ідеал $I$ напівгрупи $S$ породжує на ній конгруенцію Ріса: $\mathfrak{K}_I=(I\times I)\cup\Delta_S$.

Якщо $S$~--- напівгрупа, то на $E(S)$ визначено частковий порядок:
\begin{equation*}
    e\leqslant f \quad \hbox{ тоді і лише тоді, коли } \quad ef=fe=e.
\end{equation*}
Так означений частковий порядок на $E(S)$ називається \emph{природним}.

Означимо відношення $\leqslant$ на інверсній напівгрупі $S$ наступним чином:
\begin{equation*}
    s\leqslant t \qquad \hbox{тоді і лише тоді, коли}\qquad s=te.
\end{equation*}
для деякого ідемпотента $e\in S$. Так означений частковий порядок називається \emph{при\-род\-ним част\-ковим порядком} на інверсній напівгрупі $S$~\cite{Lawson1998}. Очевидно, що звуження природного часткового порядку $\leqslant$ на інверсній напівгрупі $S$ на її в'язку $E(S)$ є при\-род\-ним частковим порядком на $E(S)$. Інверсна напівгрупа $S$ називається \emph{факторизовною}, якщо для кожного елемента $s\in S$ існує елемент $g$ групи одиниць напівгрупи $S$ такий, що $s\leqslant g$ стосовно природного часткового порядку $\leqslant$ на $S$.

Надалі, через $\mathscr{P\!\!H}^+_{\!\!\operatorname{\textsf{cf}}}\!(\mathbb{R})$ ми позначатимемо множину усіх монотонних ко-скінченних часткових гомеоморфізмів топологічного простору $\mathbb{R}$.

Оскільки для довільного елемента $\alpha\in \mathscr{P\!\!H}^+_{\!\!\operatorname{\textsf{cf}}}\!(\mathbb{R})$ і для довільної підмножини $A\subseteq \operatorname{dom}\alpha$ звуження $\alpha|_{\operatorname{dom}\alpha\setminus A}\colon \operatorname{dom}\alpha\setminus A\rightarrow \operatorname{ran}\alpha\setminus (A)\alpha$ є частковим гомеоморфізмом простору $\mathbb{R}$, обернене часткове відображення $\alpha^{-1}$ до $\alpha$ існує, і $\alpha^{-1}\in \mathscr{P\!\!H}^+_{\!\!\operatorname{\textsf{cf}}}\!(\mathbb{R})$, то виконується наступне твердження.

\begin{proposition}\label{proposition-1.1}
Множина $\mathscr{P\!\!H}^+_{\!\!\operatorname{\textsf{cf}}}\!(\mathbb{R})$ з операцією композиція часткових відображень є ін\-верс\-ним підмоноїдом симетричного інверсного моноїда $\mathscr{I}\!(\mathbb{R})$ над множиною дійсних чисел $\mathbb{R}$.
\end{proposition}

Надалі одиницю та групу одиниць напівгрупи $\mathscr{P\!\!H}^+_{\!\!\operatorname{\textsf{cf}}}\!(\mathbb{R})$ будемо позначати через $1$ та $H(1)$, відповідно. З означення напівгрупи $\mathscr{P\!\!H}^+_{\!\!\operatorname{\textsf{cf}}}\!(\mathbb{R})$ випливає, що її група одиниць $H(1)$ ізоморфна групі $\mathscr{H}^+\!(\mathbb{R})$ усіх гомеоморфізмів, що зберігають орієнтацію (монотонних гомеоморфізмів) простору $\mathbb{R}$, причому група $\mathscr{H}^+\!(\mathbb{R})$ є простою (див. \cite[Наслідок~1]{Anderson1958}).

\section{Мотивація досліджень і коротка історична довідка}

Дослідження автоморфізмів і груп автоморфізмів многовидів малої розмірності фор\-му\-ють широку область сучасної математики, яка дуже швидко розвивається та роз\-та\-шо\-ва\-на на стику топології, алгебри та теорії динамічних систем. Ця область охоплює вивчення груп гомеоморфізмів прямої та кола, теорію автоморфізмів поверхонь і теорію груп класів відображень, найважливішим частковим випадком яких є групи кос Артіна, — в силу чого вказана область тісно пов'язана практично з усіма розділами мало\-ви\-мір\-ної топології (у першу чергу — с теорією вузлів і зачеплень), з диференціальною та гіперболічною геометрією, теорією ламинацій та теорією Тайхмюллєра, з комбінатор\-ною та геометричною теорією груп, теорією впорядкованих груп, і навіть з крипто\-гра\-фі\-єю.

Автоморфізмам і групам автоморфізмів многовидів розмірності 1 і 2 присвячені фундаментальні праці Клейна, Фрике, Пуанкаре, Гурвіца, Дена, Данжуа, Александера, Нільсена, Артіна, Керек'ярто, А.А.Маркова. Пізніше вказаною проблематикою зай\-ма\-ли\-ся В. Магнус, В. Бурау, Дж. Бірман, X. Цішанг, В. І. Арнольд, Г. А. Маргуліс, У. Тьорстон, О. Я. Віро, Ф. Гарсайд, В. Джонс, Е. Гіз і багато інших. В останнє десятиліття в цій області отримані такі важливі результати, як розв'язок С. Керкхофом проблеми Нільсена про реалізацію, відкриття порядку Деорнуа, доведення лінійності груп кос (Д. Краммер, С. Бігелоу) та інші. Розв'язок схожого типу  проблем потребує різноманітної техніки, а нові досягнення теорії (груп) автоморфізмів застосовуються (і, як правило, мають суттєві наслідки) в суміжних областях математики.

Так, зокрема сучасні дослідження груп гомеоморфізмів прямої викладено в оглядах Бекларяна \cite{Beklaryan-2004, Beklaryan-2015} та її застосування в теорії динамічних систем у монографії \cite{Katok-Hasselblatt-1995}.

Основні результати теорії напівгруп перетворень отримані в період
50-70-х років минулого століття викладені в оглядах Меггіла
\cite{Magill1975} та Глускіна, Шайна, Шнепермана та Ярокера
\cite{Gluskin1977a}. У цьому напрямку працювали такі відомі
математики, як Гауі, Гельфанд, Глускін, Грін, Енгелькінг, Кліффорд,
Ляпін, Меггіл, Престон, Саббах, Серпінський, Сушкевич, Улам,  Шайн,
Шнеперман, Шутов, Ярокер. На думку Меггіла (див. \cite{Magill1975}) теорія напівгруп
неперервних перетворень топологічних просторів бере свій початок з
робіт Глускіна \cite{Gluskin1959, Gluskin1959a, Gluskin1959b,
Gluskin1959c, Gluskin1960, Gluskin1960a}. В основному ці праці
Глускіна присвячені описанню структури напівгрупи $S(I)$ неперервних
перетворень одиничного відрізка $I$, а також описанню піднапівгруп
напівгрупи $S(X)$ неперервних перетворень топологічного прос\-то\-ру
$X$. Напівгрупа $S(I)$ неперервних перетворень одиничного відрізка
також дос\-лід\-жу\-ва\-лась Шутовим в працях \cite{Shutov1963a,
Shutov1963b}, де він описав максимальну власну конгруенцію на
$S(I)$.

Напівгрупа $S(I)$ також досліджувалась в працях \cite{Insaridze,
Shneperman1962a, Shneperman1963, Shneperman1965, Shneperman1966,
Cezus, Jarnik, Mioduszewski, Rosicky1974, Rosicky1974a}, зокрема в
роботах \cite{Gluskin1959, Insaridze} були описані конгруенц-прості
піднапівгрупи в $S(I)$. Шне\-пер\-ман \cite{Shneperman1965a} та Уарндоф
\cite{Warndof} показали, що одиничний відрізок визначається
напівгрупою неперервних перетворень. Інші класи топологічних
просторів, що визначаються своїми напівгрупами неперервних
перетворень були описані Уарндофом в \cite{Warndof} і Росіцким в
\cite{Rosicky1974, Rosicky1974a}. Зокрема такими є: локально зв'язні
сепарабельні метричні континууми, локально евклідові гаусдорфові
простори, нульвимірні метричні простори, $CW$-комплекси та інші.
Також О'Рейлі в праці \cite{Reilly} довела, що кожен гаусдорфовий
простір $X$ визначається напівгрупою усіх компактних відношень на
$X$.

Зауважимо, що група гомеоморфізмів дійсної прямої ізоморфна групі го\-мео\-мор\-фіз\-мів одиничного відрізка (інтервалу). Таким чином виникає задача: \emph{описати струк\-ту\-ру напівгрупи часткових
гомеоморфізмів топологічного простору $X$}, а в частковому
ви\-пад\-ку одиничного відрізка, чи дійсної прямої. Однією з останніх робіт з цієї тематики є стаття Чучмана \cite{Chuchman2011ADM}, в якій описано структуру напівгрупи замкнених зв'язних част\-ко\-вих гомеоморфізмів
одиничного відрізка з однією нерухомою точкою.

У даній праці досліджується структура напівгрупи $\mathscr{P\!\!H}^+_{\!\!\operatorname{\textsf{cf}}}\!(\mathbb{R})$ усіх монотонних ко-скінченних часткових гомеоморфізмів дійсної прямої $\mathbb{R}$. Доведено, що інверсна на\-пів\-гру\-па $\mathscr{P\!\!H}^+_{\!\!\operatorname{\textsf{cf}}}\!(\mathbb{R})$ є фак\-то\-ри\-зов\-ною та $F$-інверсною. Описано структуру в'язки на\-пів\-гру\-пи $\mathscr{P\!\!H}^+_{\!\!\operatorname{\textsf{cf}}}\!(\mathbb{R})$, її двобічні ідеали, максимальні підгрупи та відношення Гріна на ній. Доведено, що фактор-напівгрупа $\mathscr{P\!\!H}^+_{\!\!\operatorname{\textsf{cf}}}\!(\mathbb{R})/\mathfrak{C}_{\textsf{mg}}$ за найменшою груповою конгруенцією $\mathfrak{C}_{\textsf{mg}}$ ізоморфна групі $\mathscr{H}^+\!(\mathbb{R})$ усіх гомеоморфізмів, що зберігають орієнтацію простору $\mathbb{R}$, а також, що напівгрупа $\mathscr{P\!\!H}^+_{\!\!\operatorname{\textsf{cf}}}\!(\mathbb{R})$ ізоморфна напівпрямому добутку $\mathscr{H}^+\!(\mathbb{R})\ltimes_\mathfrak{h}\mathscr{P}_{\!\infty}(\mathbb{R})$ вільної напівгратки з одиницею $(\mathscr{P}_{\!\infty}(\mathbb{R}),\cup)$ з групою $\mathscr{H}^+\!(\mathbb{R})$.

\section{Алгебраїчні властивості напівгрупи $\mathscr{P\!\!H}^+_{\!\!\operatorname{\textsf{cf}}}\!(\mathbb{R})$}

Наступне твердження описує в'язку напівгрупи $\mathscr{P\!\!H}^+_{\!\!\operatorname{\textsf{cf}}}\!(\mathbb{R})$.

\begin{proposition}\label{proposition-2.1}
\begin{itemize}
  \item[$(i)$] Елемент $\varepsilon$ напівгрупи $\mathscr{P\!\!H}^+_{\!\!\operatorname{\textsf{cf}}}\!(\mathbb{R})$ є ідемпотентом тоді і лише тоді, коли $\varepsilon$~--- тотожне відображення ко-скінченної підмножини в $\mathbb{R}$.
  \item[$(ii)$] В'язка $E(\mathscr{P\!\!H}^+_{\!\!\operatorname{\textsf{cf}}}\!(\mathbb{R}))$ ізоморфна вільній напівгратці з одиницею $(\mathscr{P}_{\!\infty}(\mathbb{R}),\cup)$ сто\-сов\-но відображення $\mathfrak{h}\colon E(\mathscr{P\!\!H}^+_{\!\!\operatorname{\textsf{cf}}}\!(\mathbb{R})) \rightarrow (\mathscr{P}_{\!\infty}(\mathbb{R}),\cup)$, означеного за формулою $(\varepsilon)\mathfrak{h}=\mathbb{R}\setminus \operatorname{dom}\varepsilon$.
  \item[$(iii)$] Кожен максимальний ланцюг напівгратки $E(\mathscr{P\!\!H}^+_{\!\!\operatorname{\textsf{cf}}}\!(\mathbb{R}))$ є $\omega$-ланцюгом.
  \item[$(iv)$] Кожен максимальний антиланцюг напівгратки $E(\mathscr{P\!\!H}^+_{\!\!\operatorname{\textsf{cf}}}\!(\mathbb{R}))$, що не містить її оди\-ни\-цю має потужність континуум.
\end{itemize}
\end{proposition}

\begin{proof}
Твердження $(i)$ випливає з твердження~\ref{proposition-1.1}. Твердження $(ii)$ є очевидним, і з нього випливають $(iii)$ і $(iv)$.
\end{proof}

Оскільки при гомеоморфізмі зберігається кількість компонент зв'язності, то ви\-ко\-ну\-єть\-ся наступне твердження.

\begin{proposition}\label{proposition-2.2}
Якщо $\alpha\in \mathscr{P\!\!H}^+_{\!\!\operatorname{\textsf{cf}}}\!(\mathbb{R})$, то $|\mathbb{R}\setminus\operatorname{dom}\alpha|= |\mathbb{R}\setminus\operatorname{ran}\alpha|$. Більше того, якщо $\alpha$~--- част\-ковий монотонний гомеоморфізм простору $\mathbb{R}$ такий, що $|\mathbb{R}\setminus\operatorname{dom}\alpha|= |\mathbb{R}\setminus\operatorname{ran}\alpha|<\omega$, то $\alpha\in \mathscr{P\!\!H}^+_{\!\!\operatorname{\textsf{cf}}}\!(\mathbb{R})$.
\end{proposition}

Наступне твердження випливає з означення напівгрупи $\mathscr{P\!\!H}^+_{\!\!\operatorname{\textsf{cf}}}\!(\mathbb{R})$.

\begin{proposition}\label{proposition-2.3}
Для довільних елементів $\alpha$ i $\beta$ напівгрупи
$\mathscr{P\!\!H}^+_{\!\!\operatorname{\textsf{cf}}}\!(\mathbb{R})$ виконується умо\-ва
\begin{equation*}
    |\mathbb{R}\setminus\operatorname{dom}(\alpha\beta)|\leqslant
    |\mathbb{R}\setminus\operatorname{dom}\alpha|+
    |\mathbb{R}\setminus\operatorname{dom}\beta|.
\end{equation*}
\end{proposition}

\begin{lemma}\label{lemma-2.4}
Нехай $n$~--- довільне невід'ємне ціле число. Тоді для довільних елементів $\alpha$ i $\beta$ напівгрупи $\mathscr{P\!\!H}^+_{\!\!\operatorname{\textsf{cf}}}\!(\mathbb{R})$ таких, що $|\mathbb{R}\setminus\operatorname{dom}\alpha|=          |\mathbb{R}\setminus\operatorname{dom}\beta|=n$ існують елементи $\gamma$ i $\delta$ в $\mathscr{P\!\!H}^+_{\!\!\operatorname{\textsf{cf}}}\!(\mathbb{R})$ такі, що $\alpha=\gamma\cdot\beta\cdot\delta$ і          $|\mathbb{R}\setminus\operatorname{dom}\gamma|=          |\mathbb{R}\setminus\operatorname{dom}\delta|=n$.
\end{lemma}

\begin{proof}
Нехай $\alpha$ i $\beta$~--- елементи напівгрупи $\mathscr{P\!\!H}^+_{\!\!\operatorname{\textsf{cf}}}\!(\mathbb{R})$ такі, що $|\mathbb{R}\setminus\operatorname{dom}\alpha|=          |\mathbb{R}\setminus\operatorname{dom}\beta|=n$. У випадку $n=0$ маємо, що $\alpha$ i $\beta$~--- елементи групи одиниць $H(1)$ напівгрупи $\mathscr{P\!\!H}^+_{\!\!\operatorname{\textsf{cf}}}\!(\mathbb{R})$, а отже існують елементи $\gamma,\delta\in \mathscr{H}^+\!(\mathbb{R})\subseteq \mathscr{P\!\!H}^+_{\!\!\operatorname{\textsf{cf}}}\!(\mathbb{R})$ такі, що $\alpha=\gamma\cdot\beta\cdot\delta$ і          $|\mathbb{R}\setminus\operatorname{dom}\gamma|=          |\mathbb{R}\setminus\operatorname{dom}\delta|=0$. Тому, надалі будемо вважати, що $n>0$. Нехай $\mathbb{R}\setminus\operatorname{dom}\alpha=\{a_1,\ldots,a_n\}$ i $\mathbb{R}\setminus\operatorname{dom}\beta=\{b_1,\ldots,b_n\}$, причому $a_1<\ldots<a_n$ i $b_1<\ldots<b_n$. Через $\eta_0$ позначимо довільний гомеоморфізм інтервала $(-\infty,a_1)$ на інтервал $(-\infty,b_1)$, через $\eta_n$ позначимо довільний гомеоморфізм інтервала $(a_n,+\infty)$ на інтервал $(b_n,+\infty)$, і для довільного натурального числа $i=1,\ldots,n-1$ --- через $\eta_i$ позначимо довільний гомеоморфізм інтервала $(a_i,a_{i+1})$ на інтервал $(b_i,b_{i+1})$. Такі гомеоморфізми (навіть як лінійні відображення) існують, оскільки кожні такі два інтервали є попарно гомеоморфними \cite{Engelking1989}. Покладемо $\eta=\eta_0\cup\eta_1\cup\cdots\cup\eta_n$. Тоді, очевидно, що $\eta\in \mathscr{P\!\!H}^+_{\!\!\operatorname{\textsf{cf}}}\!(\mathbb{R})$, $\alpha\alpha^{-1}=\eta\eta^{-1}$ i $\eta^{-1}\eta=\beta\beta^{-1}$. Таким чином, отримуємо
\begin{equation*}
  \alpha =\alpha\alpha^{-1}\alpha=\eta\eta^{-1}\alpha= (\eta\eta^{-1})(\eta\eta^{-1})\alpha= \eta(\eta^{-1}\eta)\eta^{-1}\alpha=
  \eta(\beta\beta^{-1})\eta^{-1}\alpha= \eta\beta(\beta^{-1}\eta^{-1}\alpha).
\end{equation*}
З побудови гомеоморфізму $\eta$ випливає, що $\operatorname{ran}\alpha=\operatorname{dom}\eta$ i $\operatorname{dom}\beta=\operatorname{ran}\eta$. Звідси от\-ри\-му\-ємо, що $\operatorname{ran}(\beta^{-1})=\operatorname{dom}(\eta^{-1})= \operatorname{dom}\alpha$, а отже $\operatorname{ran}(\beta^{-1}\eta^{-1}\alpha)= \operatorname{ran}\alpha$ і $\operatorname{dom}(\beta^{-1}\eta^{-1}\alpha)= \operatorname{ran}\beta$. Таким чином, елементи $\gamma=\eta$ i $\delta=\beta^{-1}\eta^{-1}\alpha$~--- шукані.
\end{proof}

Наступне твердження випливає з твердження~\ref{proposition-2.3} та леми~\ref{lemma-2.4}, і воно описує двобічні ідеали напівгрупи $\mathscr{P\!\!H}^+_{\!\!\operatorname{\textsf{cf}}}\!(\mathbb{R})$.

\begin{proposition}\label{proposition-2.5}
Для довільного невід'ємного цілого числа $n$ множина
\begin{equation*}
I_n= \{\alpha\in \mathscr{P\!\!H}^+_{\!\!\operatorname{\textsf{cf}}}\!(\mathbb{R}) \mid
|\mathbb{R}\setminus\operatorname{dom}\alpha|\geqslant n\}
\end{equation*}
є двобічним ідеалом в
$\mathscr{P\!\!H}^+_{\!\!\operatorname{\textsf{cf}}}\!(\mathbb{R})$. Більше того, $I_0=\mathscr{P\!\!H}^+_{\!\!\operatorname{\textsf{cf}}}\!(\mathbb{R})$ i для кожного власного двобічного ідеалу $I$ напівгрупи $\mathscr{P\!\!H}^+_{\!\!\operatorname{\textsf{cf}}}\!(\mathbb{R})$ існує натуральне число $n$ таке, що ідеали $I$ та $I_n$ є ізоморфними.
\end{proposition}

Якщо $S$ --- напівгрупа, то через $\mathscr{R}$, $\mathscr{L}$, $\mathscr{D}$, $\mathscr{H}$ і $\mathscr{J}$ позначаються відношення Гріна на $S$ (див.
\cite[\S2.1]{CliffordPreston1961-1967}):
\begin{align*}
    &\qquad a\mathscr{R}b \quad\mbox{ тоді і лише тоді, коли }\quad aS^1=bS^1;\\
    &\qquad a\mathscr{L}b \quad\mbox{ тоді і лише тоді, коли }\quad S^1a=S^1b;\\
    &\qquad \mathscr{D}=\mathscr{L}\circ\mathscr{R}=\mathscr{R}\circ\mathscr{L};\\
    &\qquad \mathscr{H}=\mathscr{L}\cap\mathscr{R};\\
    &\qquad a\mathscr{J}b \quad\mbox{ тоді і лише тоді, коли }\quad S^1aS^1=S^1bS^1.
\end{align*}

Наступна теорема описує відношення Гріна на напівгрупі $\mathscr{P\!\!H}^+_{\!\!\operatorname{\textsf{cf}}}\!(\mathbb{R})$.

\begin{theorem}\label{theorem-2.6}
Нехай $\alpha,\beta\in \mathscr{P\!\!H}^+_{\!\!\operatorname{\textsf{cf}}}\!(\mathbb{R})$. Тоді:
\begin{itemize}
  \item[$(i)$] $\alpha\mathscr{R}\beta$ тоді і лише тоді, коли $\operatorname{dom}\alpha=\operatorname{dom}\beta$;
  \item[$(ii)$] $\alpha\mathscr{L}\beta$ тоді і лише тоді, коли $\operatorname{ran}\alpha=\operatorname{ran}\beta$;
  \item[$(iii)$] $\alpha\mathscr{H}\beta$ тоді і лише тоді, коли $\operatorname{dom}\alpha=\operatorname{dom}\beta$ i $\operatorname{ran}\alpha=\operatorname{ran}\beta$;
  \item[$(iv)$] $\alpha\mathscr{D}\beta$ тоді і лише тоді, коли $|\mathbb{R}\setminus\operatorname{dom}\alpha|= |\mathbb{R}\setminus\operatorname{dom}\beta|$;
  \item[$(v)$] $\mathscr{J}=\mathscr{D}$ в $\mathscr{P\!\!H}^+_{\!\!\operatorname{\textsf{cf}}}\!(\mathbb{R})$.
\end{itemize}
\end{theorem}

\begin{proof}
$(i)$ Нехай $\alpha$ i $\beta$~--- елементи напівгрупи $\mathscr{P\!\!H}^+_{\!\!\operatorname{\textsf{cf}}}\!(\mathbb{R})$ такі, що
$\alpha\mathscr{R}\beta$. Позаяк $\mathscr{P\!\!H}^+_{\!\!\operatorname{\textsf{cf}}}\!(\mathbb{R})$~--- інверсна напівгрупа і $\alpha\mathscr{P\!\!H}^+_{\!\!\operatorname{\textsf{cf}}}\!(\mathbb{R})=
\beta\mathscr{P\!\!H}^+_{\!\!\operatorname{\textsf{cf}}}\!(\mathbb{R})$, то з теореми~1.17
\cite{CliffordPreston1961-1967} ви\-пли\-ває, що $\alpha\mathscr{P\!\!H}^+_{\!\!\operatorname{\textsf{cf}}}\!(\mathbb{R})=
\alpha\alpha^{-1}\mathscr{P\!\!H}^+_{\!\!\operatorname{\textsf{cf}}}\!(\mathbb{R})$ і
$\beta\mathscr{P\!\!H}^+_{\!\!\operatorname{\textsf{cf}}}\!(\mathbb{R})=
\beta\beta^{-1}\mathscr{P\!\!H}^+_{\!\!\operatorname{\textsf{cf}}}\!(\mathbb{R})$, а отже $\alpha\alpha^{-1}=\beta\beta^{-1}$. Таким чином, виконується рівність
$\operatorname{dom}\alpha=\operatorname{dom}\beta$.

Навпаки, нехай $\alpha$ i $\beta$~--- елементи напівгрупи $\mathscr{P\!\!H}^+_{\!\!\operatorname{\textsf{cf}}}\!(\mathbb{R})$ такі, що $\operatorname{dom}\alpha=\operatorname{dom}\beta$. Тоді
$\alpha\alpha^{-1}=\beta\beta^{-1}$. Позаяк
$\mathscr{P\!\!H}^+_{\!\!\operatorname{\textsf{cf}}}\!(\mathbb{R})$~--- інверсна напівгрупа, то з теореми~1.17
\cite{CliffordPreston1961-1967} випливає, що  $\alpha\mathscr{P\!\!H}^+_{\!\!\operatorname{\textsf{cf}}}\!(\mathbb{R})=
\alpha\alpha^{-1}\mathscr{P\!\!H}^+_{\!\!\operatorname{\textsf{cf}}}\!(\mathbb{R})=
\beta\mathscr{P\!\!H}^+_{\!\!\operatorname{\textsf{cf}}}\!(\mathbb{R})$, а отже виконується рівність
$\alpha\mathscr{P\!\!H}^+_{\!\!\operatorname{\textsf{cf}}}\!(\mathbb{R})=
\beta\mathscr{P\!\!H}^+_{\!\!\operatorname{\textsf{cf}}}\!(\mathbb{R})$.

Доведення $(ii)$ аналогічне доведенню пункту $(i)$.

Пункт $(iii)$ випливає з $(i)$ та $(ii)$.

$(iv)$ Нехай $\alpha$ i $\beta$~--- елементи напівгрупи $\mathscr{P\!\!H}^+_{\!\!\operatorname{\textsf{cf}}}\!(\mathbb{R})$ такі, що
$\alpha\mathscr{D}\beta$. Тоді існує елемент $\gamma$ напівгрупи $\mathscr{P\!\!H}^+_{\!\!\operatorname{\textsf{cf}}}\!(\mathbb{R})$ такий, що $\alpha\mathscr{R}\gamma$ i $\gamma\mathscr{L}\beta$. За твердженнями~$(i)$ та $(ii)$ маємо, що $\operatorname{dom}\alpha=\operatorname{dom}\gamma$ та $\operatorname{ran}\gamma=\operatorname{ran}\beta$. Тоді, використавши твердження~\ref{proposition-2.2}, отримуємо
\begin{equation*}
|\mathbb{R}\setminus\operatorname{dom}\alpha|= |\mathbb{R}\setminus\operatorname{dom}\gamma|= |\mathbb{R}\setminus\operatorname{ran}\gamma|= |\mathbb{R}\setminus\operatorname{ran}\beta|= |\mathbb{R}\setminus\operatorname{dom}\beta|.
\end{equation*}

Навпаки, нехай $\alpha$ i $\beta$~--- елементи напівгрупи $\mathscr{P\!\!H}^+_{\!\!\operatorname{\textsf{cf}}}\!(\mathbb{R})$ такі, що $|\mathbb{R}\setminus\operatorname{dom}\alpha|= |\mathbb{R}\setminus\operatorname{dom}\beta|=n$, де $n$~--- деяке невід'ємне ціле число. У випадку $n=0$ маємо, що $\alpha$ i $\beta$~--- елементи групи одиниць напівгрупи $\mathscr{P\!\!H}^+_{\!\!\operatorname{\textsf{cf}}}\!(\mathbb{R})$, тобто $\alpha$ i $\beta$ є гомеоморфізмами простору $\mathbb{R}$, а отже вони є $\mathscr{D}$-еквівалентними, оскільки $\alpha\mathscr{H}\beta$. Тому, надалі будемо вважати, що $n>0$. Нехай $\mathbb{R}\setminus\operatorname{dom}\alpha=\{a_1,\ldots,a_n\}$ i $\mathbb{R}\setminus\operatorname{dom}\beta=\{b_1,\ldots,b_n\}$, причому $a_1<\ldots<a_n$ i $b_1<\ldots<b_n$. Через $\gamma_0$ позначимо довільний гомеоморфізм інтервала $(-\infty,a_1)$ на інтервал $(-\infty,b_1)$, через $\gamma_n$ позначимо довільний гомеоморфізм інтервала $(a_n,+\infty)$ на інтервал $(b_n,+\infty)$, і для довільного натурального числа $i=1,\ldots,n-1$ --- через $\gamma_i$ позначимо довільний гомеоморфізм інтервала $(a_i,a_{i+1})$ на інтервал $(b_i,b_{i+1})$. Такі гомеоморфізми (навіть як лінійні відображення) існують, оскільки кожні такі два інтервали попарно гомеоморфні \cite{Engelking1989}. Покладемо $\gamma=\gamma_0\cup\gamma_1\cup\cdots\cup\gamma_n$. Тоді, очевидно, що $\gamma\in \mathscr{P\!\!H}^+_{\!\!\operatorname{\textsf{cf}}}\!(\mathbb{R})$, $\alpha\mathscr{R}\gamma$ i $\gamma\mathscr{L}\beta$, а отже $\alpha\mathscr{D}\beta$.

$(v)$ Включення $\mathscr{D}\subseteq\mathscr{J}$ випливає з означень відношень Гріна $\mathscr{D}$ і $\mathscr{J}$. Обернене включення випливає з пункту $(iv)$, леми~\ref{lemma-2.4} та твердження~\ref{proposition-2.5}.
\end{proof}

\begin{lemma}\label{lemma-2.7}
Для довільного елемента $\alpha$ напівгрупи $\mathscr{P\!\!H}^+_{\!\!\operatorname{\textsf{cf}}}\!(\mathbb{R})$ існує єдиний елемент $\gamma_\alpha$ групи одиниць $H(1)$ такий, що $\alpha=\alpha\alpha^{-1}\gamma_\alpha= \gamma_\alpha\alpha^{-1}\alpha$.
\end{lemma}

\begin{proof}
У випадку $\alpha\in H(1)$ твердження леми є очевидним, а тому далі вва\-жа\-ти\-ме\-мо, що $\alpha\notin H(1)$. Нехай $\operatorname{dom}\alpha=\mathbb{R}\setminus\{x_1,\ldots,x_n\}$ i $\operatorname{ran}\alpha=\mathbb{R}\setminus\{y_1,\ldots,y_n\}$, причому $x_1<\ldots<x_n$ i $y_1<\ldots<y_n$ в $\mathbb{R}$. Означимо відображення $\gamma_\alpha\colon\mathbb{R}\to\mathbb{R}$ за формулою
\begin{equation}\label{gamma_alpha}
    (x)\gamma_\alpha=
\left\{
  \begin{array}{cl}
    (x)\alpha, & \hbox{якщо~}~ x\in\operatorname{dom}\alpha; \\
    y_1,       & \hbox{якщо~}~ x=x_1; \\
    \vdots     & \qquad \vdots  \\
    y_n,       & \hbox{якщо~}~ x=x_n.
  \end{array}
\right.
\end{equation}
Тоді, очевидно, що так означене відображення $\gamma_\alpha$ є монотонним гомеоморфізмом прос\-то\-ру $\mathbb{R}$, та оскільки $\alpha\alpha^{-1}$ i $\alpha^{-1}\alpha$~--- тотожні відображення множин $\operatorname{dom}\alpha=\mathbb{R}\setminus\{x_1,\ldots,x_n\}$ i $\operatorname{ran}\alpha=\mathbb{R}\setminus\{y_1,\ldots,y_n\}$, відповідно, то отримуємо, що $\alpha=\alpha\alpha^{-1}\gamma_\alpha= \gamma_\alpha\alpha^{-1}\alpha$.

Єдиність елемента $\gamma_\alpha$ випливає з означення напівгрупи $\mathscr{P\!\!H}^+_{\!\!\operatorname{\textsf{cf}}}\!(\mathbb{R})$.
\end{proof}

З леми~\ref{lemma-2.7} випливає

\begin{corollary}\label{corollary-2.8}
$\mathscr{P\!\!H}^+_{\!\!\operatorname{\textsf{cf}}}\!(\mathbb{R})$~--- факторизовна інверсна напівгрупа.
\end{corollary}

\begin{construction}\label{construction-2.9}
Нехай $\varepsilon$~--- довільний, відмінний від одиничного елемента, ідемпотент напівгрупи
$\mathscr{P\!\!H}^+_{\!\!\operatorname{\textsf{cf}}}\!(\mathbb{R})$. Тоді за твердженням~\ref{proposition-2.1}$(i)$ існують дійсні числа $x_1,\ldots,x_n$, $n\in\mathbb{N}$, такі, що $\operatorname{dom}\varepsilon=\operatorname{ran}\varepsilon= \mathbb{R}\setminus\{x_1,\ldots,x_n\}$, $\varepsilon\colon \mathbb{R}\setminus\{x_1,\ldots,x_n\}\to\mathbb{R}\setminus\{x_1,\ldots,x_n\}$~--- тотожне відображення і $x_1<\ldots<x_n$. З теореми~\ref{theorem-2.6}$(iii)$ випливає, що для кожного елементу $\alpha$ $\mathscr{H}$-класу $H(\varepsilon)$, який містить ідемпотент $\varepsilon$, виконується умова $\operatorname{dom}\alpha=\operatorname{ran}\alpha= \mathbb{R}\setminus\{x_1,\ldots,x_n\}$. Позаяк для довільних різних дійсних чисел $a$, $b$ i $c$ інтервали $(a,b)$, $(-\infty,c)$ i $(c,+\infty)$ є гомеоморфними, то надалі для спрощення викладу покладемо $x_0=-\infty$ i $x_{n+1}=+\infty$. Для довільного числа $i\in\{0,1,\ldots,n\}$ означимо часткове відображення $\alpha^{(x_i,x_{i+1})}\colon \mathbb{R}\to\mathbb{R}$ наступним чином: $\operatorname{dom}\alpha^{(x_i,x_{i+1})}=\operatorname{ran}\alpha^{(x_i,x_{i+1})}= \mathbb{R}\setminus\{x_1,\ldots,x_n\}$ і
\begin{equation*}
(x)\alpha^{(x_i,x_{i+1})}=
\left\{
  \begin{array}{cl}
    (x)\alpha, & \hbox{якщо~}~ x\in (x_i,x_{i+1});\\
    x, & \hbox{якщо~}~ x\notin (x_i,x_{i+1}).
  \end{array}
\right.
\end{equation*}
Тоді очевидно, що $\alpha^{(x_i,x_{i+1})}\in \mathscr{P\!\!H}^+_{\!\!\operatorname{\textsf{cf}}}\!(\mathbb{R})$.
\end{construction}

В термінах попередньої конструкції зробимо наступне зауваження.

\begin{remark}\label{remark-2.10}
\begin{itemize}
  \item[$(i)$] За теоремою~\ref{theorem-2.6}$(iii)$ маємо, що $\alpha^{(x_i,x_{i+1})}\in H(\varepsilon)$ для довільних $\alpha\in H(\varepsilon)$ та $i\in\{0,1,\ldots,n\}$.
  \item[$(ii)$] Якщо $\alpha$~--- ідемпотент, то $\alpha^{(x_i,x_{i+1})}=\alpha$, для довільного числа $i\in\{0,1,\ldots,n\}$.
  \item[$(iii)$] $\alpha=\alpha^{(x_0,x_{1})}\cdot\ldots \cdot\alpha^{(x_n,x_{n+1})}$ для довільного елемента $\alpha\in H(\varepsilon)$, причому це зображення єдине з точністю до перестановки множників і множення на ідемпотент $\varepsilon_0\geqslant\varepsilon$ напівгрупи $\mathscr{P\!\!H}^+_{\!\!\operatorname{\textsf{cf}}}\!(\mathbb{R})$.
  \item[$(iv)$] Оскільки довільний інтервал дійсної прямої гомеоморфний цій прямій, то для довільного числа $i\in\{0,1,\ldots,n\}$ підмножина $\left\{\alpha^{(x_i,x_{i+1})}\colon \alpha\in H(\varepsilon)\right\}$ з індукованою з $\mathscr{P\!\!H}^+_{\!\!\operatorname{\textsf{cf}}}\!(\mathbb{R})$ бінарною операцією є групою, яка ізоморфна групі $\mathscr{H}^+\!(\mathbb{R})$.
\end{itemize}
\end{remark}

Наслідок~\ref{corollary-2.11} описує структуру максимальних підгруп напівгрупи $\mathscr{P\!\!H}^+_{\!\!\operatorname{\textsf{cf}}}\!(\mathbb{R})$ і він випливає з пунктів $(iii)$ та $(iv)$ зауваження~\ref{remark-2.10}.

\begin{corollary}\label{corollary-2.11}
Нехай $\varepsilon$~--- неодиничний ідемпотент напівгрупи $\mathscr{P\!\!H}^+_{\!\!\operatorname{\textsf{cf}}}\!(\mathbb{R})$. Тоді $\mathscr{H}$-клас $H(\varepsilon)$, який містить ідемпотент $\varepsilon$, ізоморфний прямому степеню $\left(\mathscr{H}^+\!(\mathbb{R})\right)^{n+1}$, де $n=|\mathbb{R}\setminus\operatorname{dom}\varepsilon|$.
\end{corollary}

Виконується наступне твердження:

\begin{proposition}\label{propositiony-2.12}
Кожен елемент $\alpha$ напівгрупи $\mathscr{P\!\!H}^+_{\!\!\operatorname{\textsf{cf}}}\!(\mathbb{R})$ стосовно природного част\-ко\-во\-го порядку порівняльний з єдиним елементом $\gamma_\alpha$ групи одиниць $H(1)$.
\end{proposition}

\begin{proof}
З леми~\ref{lemma-2.7} і леми~1.4.6~\cite{Lawson1998} випливає, що для довільного елемента $\alpha$ на\-пів\-гру\-пи $\mathscr{P\!\!H}^+_{\!\!\operatorname{\textsf{cf}}}\!(\mathbb{R})$ існує елемент $\gamma_\alpha$ групи одиниць $H(1)$ такий, що $\alpha\leqslant\gamma_\alpha$.
Припустимо про\-ти\-леж\-не до твердження: існує елемент $\alpha$ напівгрупи $\mathscr{P\!\!H}^+_{\!\!\operatorname{\textsf{cf}}}\!(\mathbb{R})$ такий, що $\alpha\leqslant\gamma_\alpha$ i $\alpha\leqslant\gamma_\alpha^{\ast}$ для двох різних $\gamma_\alpha,\gamma_\alpha^{\ast}\in H(1)$. З твердження~1.4.10~\cite{Lawson1998} випливає, що $\alpha\notin H(1)$. Тоді існує дійсне число $x$ таке, що $(x)\gamma_\alpha\neq(x)\gamma_\alpha^{\ast}$. З означення напівгрупи $\mathscr{P\!\!H}^+_{\!\!\operatorname{\textsf{cf}}}\!(\mathbb{R})$ випливає, що її група одиниць $H(1)$ ізоморфна групі $\mathscr{H}^+\!(\mathbb{R})$ усіх гомеоморфізмів, що зберігають орієнтацію простору $\mathbb{R}$, а отже за гаусдорфовістю простору $\mathbb{R}$ отримуємо, що існує дійсне число $a>0$ таке, що $(y)\gamma_\alpha\neq(y)\gamma_\alpha^{\ast}$ для всіх $y\in(x-a,x+a)$. Тоді з твердження~\ref{proposition-2.1}$(i)$ випливає, що $\varepsilon\gamma_\alpha\neq\varepsilon\gamma_\alpha^{\ast}$ для довільного ідемпотента $\varepsilon$ напівгрупи $\mathscr{P\!\!H}^+_{\!\!\operatorname{\textsf{cf}}}\!(\mathbb{R})$, а це суперечить тому, що $\alpha\leqslant\gamma_\alpha$. З отриманого протиріччя випливає єдиність елемента $\gamma_\alpha$.
\end{proof}

Найменша групова конгруенція $\mathfrak{C}_{\textsf{mg}}$ на інверсній напівгрупі $S$ визначається на\-ступ\-ним чином (див. \cite[III.5]{Petrich1984}):
\begin{equation*}
    s\mathfrak{C}_{\textsf{mg}}t \hbox{~в~} S \quad \hbox{~тоді і лише тоді, коли існує ідемпотент~} \quad  e\in S \quad  \hbox{~такий, що~} \quad  es=et.
\end{equation*}

За твердженням~\ref{propositiony-2.12} для довільного елемента $\alpha$ напівгрупи $\mathscr{P\!\!H}^+_{\!\!\operatorname{\textsf{cf}}}\!(\mathbb{R})$ існує єдиний елемент $\gamma_\alpha$ групи одиниць $H(1)$ такий, що $\alpha\leqslant\gamma_\alpha$, а отже визначено відображення
\begin{equation}\label{eq-Gamma}
 \mathfrak{G}\colon \mathscr{P\!\!H}^+_{\!\!\operatorname{\textsf{cf}}}\!(\mathbb{R})\rightarrow H(1)\colon\alpha\mapsto\gamma_\alpha.
\end{equation}
Також, з тверджень~1.4.7~\cite{Lawson1998} і~\ref{propositiony-2.12} випливає, що так означене відображення $\mathfrak{G}$ є сюр'єктивним гомоморфізмом, і більше того  для $\alpha,\beta\in\mathscr{P\!\!H}^+_{\!\!\operatorname{\textsf{cf}}}\!(\mathbb{R})$ маємо, що
\begin{equation*}
    \alpha\mathfrak{C}_{\textsf{mg}}\beta  \quad \hbox{~тоді і лише тоді, коли~} \quad  (\alpha) \mathfrak{G}=(\beta) \mathfrak{G}.
\end{equation*}

Таким чином, нами доведена наступна теорема:

\begin{theorem}\label{theorem-2.13}
Фактор-напівгрупа $\mathscr{P\!\!H}^+_{\!\!\operatorname{\textsf{cf}}}\!(\mathbb{R})/\mathfrak{C}_{\textsf{mg}}$ ізоморфна групі $\mathscr{H}^+\!(\mathbb{R})$ усіх го\-мео\-мор\-фіз\-мів, що зберігають орієнтацію простору $\mathbb{R}$, причому природний гомоморфізм $\mathfrak{C}_{\textsf{mg}}^{\natural}\colon \mathscr{P\!\!H}^+_{\!\!\operatorname{\textsf{cf}}}\!(\mathbb{R})\rightarrow \mathscr{H}^+\!(\mathbb{R})$ визначається за формулою $($\ref{eq-Gamma}$)$.
\end{theorem}

Наступне твердження трішки ``посилює'' властивості напівгрупи $\mathscr{P\!\!H}^+_{\!\!\operatorname{\textsf{cf}}}\!(\mathbb{R})$ за модулем леми~\ref{lemma-2.7}.

\begin{proposition}\label{proposition-2.14}
Для довільного елемента $\alpha$ напівгрупи $\mathscr{P\!\!H}^+_{\!\!\operatorname{\textsf{cf}}}\!(\mathbb{R})$ існує єдиний еле\-мент $\gamma_\alpha$ групи одиниць $H(1)$ напівгрупи $\mathscr{P\!\!H}^+_{\!\!\operatorname{\textsf{cf}}}\!(\mathbb{R})$ такий, що
\begin{equation}\label{eq-2}
    \alpha\alpha^{-1}=\alpha\gamma_\alpha^{-1} \qquad \hbox{i} \qquad \alpha^{-1}\alpha=\gamma_\alpha^{-1}\alpha.
\end{equation}
\end{proposition}

\begin{proof}
Для елемента $\alpha$ напівгрупи $\mathscr{P\!\!H}^+_{\!\!\operatorname{\textsf{cf}}}\!(\mathbb{R})$ означимо гомеоморфізм $\widetilde{\alpha}$ простору $\mathbb{R}$ на\-ступ\-ним чином. Якщо $\alpha$ --- елемент групи одиниць $H(1)$ напівгрупи $\mathscr{P\!\!H}^+_{\!\!\operatorname{\textsf{cf}}}\!(\mathbb{R})$, то покладемо $\gamma_\alpha=\alpha$.

Нехай $\alpha\notin H(1)$ i $|\mathbb{R}\setminus\operatorname{dom}\alpha|= |\mathbb{R}\setminus\operatorname{ran}\alpha|=n\geqslant 1$. Тоді існують дійсні числа $x_1,\ldots,x_n,y_1,\ldots,y_n$ такі, що $\operatorname{dom}\alpha= \mathbb{R}\setminus\{x_1,\ldots,x_n\}$, $\operatorname{ran}\alpha=\mathbb{R}\setminus\{y_1,\ldots,y_n\}$, $x_1<\ldots<x_n$ i $y_1<\ldots<y_n$. Означимо відображення $\gamma_\alpha\colon\mathbb{R}\rightarrow\mathbb{R}$ за формулою (\ref{gamma_alpha}). Очевидно, що так означене відображення $\gamma_\alpha$ є монотонним гомеоморфізмом простору $\mathbb{R}$ та елемент $\gamma_\alpha$ групи одиниць $H(1)$ напівгрупи $\mathscr{P\!\!H}^+_{\!\!\operatorname{\textsf{cf}}}\!(\mathbb{R})$ задовольняє умову (\ref{eq-2}). Також з неперервності відображення $\alpha\colon\operatorname{dom}\alpha\rightarrow \operatorname{ran}\alpha$ та ко-скінченності множин $\operatorname{dom}\alpha$ та $\operatorname{ran}\alpha$ в $\mathbb{R}$ випливає, що $(x_1)\beta=y_1,\ldots,(x_n)\beta=y_n$, для довільного гомеоморфізму $\beta$ простору $\mathbb{R}$ такого, що звуження $\beta|_{\operatorname{dom}\alpha}$ збігається з частковим відображенням $\alpha$. Таким чином, так побудований гомеоморфізм $\gamma_\alpha\colon\mathbb{R}\rightarrow\mathbb{R}$ єдиний, що задовольняє умову (\ref{eq-2}).
\end{proof}

Далі ми опишемо додаткові властивості $\mathscr{D}$-еквівалентних елементів у напівгрупі\break $\mathscr{P\!\!H}^+_{\!\!\operatorname{\textsf{cf}}}\!(\mathbb{R})$, які пов'язані з її групою одиниць.

\begin{lemma}\label{lemma-2.15}
Для довільних  $\mathscr{D}$-еквівалентних ідемпотентів $\iota$ та $\varepsilon$ напівгрупи $\mathscr{P\!\!H}^+_{\!\!\operatorname{\textsf{cf}}}\!(\mathbb{R})$ існує елемент $\gamma_{\iota,\varepsilon}$ групи одиниць $H(1)$ напівгрупи $\mathscr{P\!\!H}^+_{\!\!\operatorname{\textsf{cf}}}\!(\mathbb{R})$ такий, що $\iota=\gamma_{\iota,\varepsilon}\varepsilon\gamma_{\iota,\varepsilon}^{-1}$.
\end{lemma}

\begin{proof}
За теоремою~\ref{theorem-2.6}$(iv)$ маємо, що  $|\mathbb{R}\setminus\operatorname{dom}\iota|= |\mathbb{R}\setminus\operatorname{dom}\varepsilon|$. Покладемо
\begin{equation*}
    \operatorname{dom}\iota=\operatorname{ran}\iota=\mathbb{R}\setminus\{x_1,\ldots,x_n\} \qquad \hbox{i} \qquad \operatorname{dom}\varepsilon=\operatorname{ran}\varepsilon=\mathbb{R}\setminus\{y_1,\ldots,y_n\}.
\end{equation*}
Не зменшуючи загальності, можемо вважати, що $x_1<\ldots<x_n$ i $y_1<\ldots<y_n$. Означимо відображення $\gamma_{\iota,\varepsilon}\colon\mathbb{R} \rightarrow \mathbb{R}$ за формулою
\begin{equation*}
    (x)\gamma_{\iota,\varepsilon}=
\left\{
  \begin{array}{ll}
    x-x_1+y_1, & \hbox{якщо~} x\leqslant x_1;\\
    \dfrac{y_2-y_1}{x_2-x_1}\cdot(x-x_1)+y_1, & \hbox{якщо~} x_1\leqslant x\leqslant x_2;\\
     \qquad \vdots & \qquad \vdots\\
    \dfrac{y_n-y_{n-1}}{x_n-x_{n-1}}\cdot(x-x_{n-1})+y_{n-1}, & \hbox{якщо~} x_{n-1}\leqslant x\leqslant x_n;\\
    x-x_n+y_n, & \hbox{якщо~} x\geqslant x_n.
  \end{array}
\right.
\end{equation*}
Легко бачити, що відображення $\gamma_{\iota,\varepsilon}\colon\mathbb{R} \rightarrow \mathbb{R}$ означене коректно, причому $(x_1)\gamma_{\iota,\varepsilon}=y_1, \ldots, (x_n)\gamma_{\iota,\varepsilon}=y_n$. Позаяк відображення $\gamma_{\iota,\varepsilon}$ є кусково-лінійним, то воно, очевидно, є гомеоморфізмом простору $\mathbb{R}$, а отже $\gamma_{\iota,\varepsilon}$ є елементом групи одиниць $H(1)$ напівгрупи $\mathscr{P\!\!H}^+_{\!\!\operatorname{\textsf{cf}}}\!(\mathbb{R})$. Також з означення відображення $\gamma_{\iota,\varepsilon}$ випливає, що виконується рівність $\iota=\gamma_{\iota,\varepsilon}\varepsilon\gamma_{\iota,\varepsilon}^{-1}$.
\end{proof}

\begin{proposition}\label{proposition-2.16}
Для довільних $\mathscr{D}$-еквівалентних елементів $\alpha$ та $\beta$ напівгрупи\break $\mathscr{P\!\!H}^+_{\!\!\operatorname{\textsf{cf}}}\!(\mathbb{R})$ існують елементи $\gamma$ та $\delta$ групи одиниць $H(1)$ напівгрупи $\mathscr{P\!\!H}^+_{\!\!\operatorname{\textsf{cf}}}\!(\mathbb{R})$ такі, що $\alpha=\gamma\beta\delta$.
\end{proposition}

\begin{proof}
За теоремою~\ref{theorem-2.6}$(iv)$ ідемпотенти $\alpha\alpha^{-1}$ та $\beta\beta^{-1}$ є $\mathscr{D}$-еквівалентними в на\-пів\-групі  $\mathscr{P\!\!H}^+_{\!\!\operatorname{\textsf{cf}}}\!(\mathbb{R})$. Тоді за лемою~\ref{lemma-2.15} існує елемент $\xi$  групи одиниць $H(1)$ напівгрупи $\mathscr{P\!\!H}^+_{\!\!\operatorname{\textsf{cf}}}\!(\mathbb{R})$ такий, що $\alpha\alpha^{-1}=\xi\beta\beta^{-1}\xi^{-1}$. З леми~\ref{lemma-2.7} випливає, що існує елемент групи одиниць $\gamma_\delta$ такий, що $\beta=\gamma_\delta\beta^{-1}\beta$. Тоді отримуємо, що $\beta^{-1}=\beta^{-1}\beta\gamma_\delta^{-1}$, а отже маємо:
\begin{equation*}
    \alpha\alpha^{-1}=\xi\beta\beta^{-1}\xi^{-1}=\xi\beta\beta^{-1}\beta\gamma_\delta^{-1}\xi^{-1}=\xi\beta\gamma_\delta^{-1}\xi^{-1}.
\end{equation*}
За твердженням~\ref{proposition-2.14} існує єдиний елемент $\gamma_\alpha$ групи одиниць $H(1)$ напівгрупи $\mathscr{P\!\!H}^+_{\!\!\operatorname{\textsf{cf}}}\!(\mathbb{R})$ такий, що $\alpha\alpha^{-1}=\alpha\gamma_\alpha^{-1}$. Таким чином, отримуємо:
\begin{equation*}
    \alpha=\alpha\cdot 1=\alpha\cdot\gamma_\alpha^{-1}\gamma_\alpha= \alpha\alpha^{-1}\gamma_\alpha= \xi\beta\gamma_\delta^{-1}\xi^{-1}\gamma_\alpha= \gamma\beta\delta,
\end{equation*}
де, очевидно, що $\gamma=\xi$ i $\delta=\gamma_\delta^{-1}\xi^{-1}\gamma_\alpha$~--- елементи групи одиниць $H(1)$ напівгрупи $\mathscr{P\!\!H}^+_{\!\!\operatorname{\textsf{cf}}}\!(\mathbb{R})$.
\end{proof}

\section{Структурна теорема для напівгрупи $\mathscr{P\!\!H}^+_{\!\!\operatorname{\textsf{cf}}}\!(\mathbb{R})$}

Нагадаємо, що інверсна напівгрупа $S$ називається \emph{$F$-інверсною}, якщо $\mathfrak{C}_{\textsf{mg}}$-клас $s_{\mathfrak{C}_{\textsf{mg}}}$ кожного елемента $s$  має найбільший елемент стосовно природного часткового порядку в $S$ \cite{McFadden-Carroll-1971}. Очевидно, що кожна $F$-інверсна напівгрупа містить одиницю.

З твердження~\ref{propositiony-2.12} випливає

\begin{corollary}\label{corollary-3.1}
$\mathscr{P\!\!H}^+_{\!\!\operatorname{\textsf{cf}}}\!(\mathbb{R})$~--- $F$-інверсна напівгрупа.
\end{corollary}

Для довільного елемента $s$ інверсної напівгрупи $S$ позначимо ${\downarrow}s=\{x\in S\colon x\leqslant s\}$, де $\leqslant$~--- природний частковий порядок на $S$.

Нехай $S$~--- довільна $F$-інверсна напівгрупа. Тоді для довільного елемента $s$ на\-пів\-гру\-пи $S$ через $e_s$ позначимо ідемпотент $ss^{-1}\in S$, через $t_s$~--- найбільший елемент стосовно природного часткового порядку в $S$ в $\mathfrak{C}_{\textsf{mg}}$-класі $s_{\mathfrak{C}_{\textsf{mg}}}$ елемента $s$, і нехай $T_S=\{t_s\colon s\in S\}$. Тоді напівгрупа $S$ є диз'юнктним об'єднанням множин ${\downarrow}t$, де $t\in T_S$ \cite{McFadden-Carroll-1971}.

Структура $F$-інверсних напівгруп викладена у наступних твердженнях з праці \cite{McFadden-Carroll-1971}, і ми далі скористаємося ними для описання напівгрупи $\mathscr{P\!\!H}^+_{\!\!\operatorname{\textsf{cf}}}\!(\mathbb{R})$.

\begin{lemma}[{\cite[лема~3]{McFadden-Carroll-1971}}]\label{lemma-3.2}
Нехай $S$~--- $F$-інверсна напівгрупа з одиницею $1_S$. Тоді:
\begin{itemize}
  \item[$(i)$] $1_S$~--- одиниця напівгратки $E(S)$;
  \item[$(ii)$] множина $T_S$ з бінарною операцією
\begin{equation*}
    u\ast v=t_{uv}, \qquad u,v\in T_S,
\end{equation*}
  є групою з нейтральним елементом $1_S$, і $t^{-1}$ є оберненим до елемента $t$ в групі $(T_S,\ast)$;
  \item[$(iii)$] для кожного елемента $t\in T_S$ відображення $\mathfrak{F}_t\colon E(S)\rightarrow {\downarrow}e_t$, означене за формулою
\begin{equation*}
    (f)\mathfrak{F}_t=tft^{-1}, \qquad f\in E(S),
\end{equation*}
   є сюр'єктивним гомоморфізмом, причому $\mathfrak{F}_{1_S}$ є тотожним відображенням на $E(S)$;
   \item[$(iv)$] $(1_S)\mathfrak{F}_t=e_t$ й $(e_t)\mathfrak{F}_{t^{-1}}=e_{t^{-1}}$, для довільного елемента $t\in T_S$;
   \item[$(v)$] для довільних елементів $u,v\in S$ виконується рівність
\begin{equation*}
   \left((1_S)\mathfrak{F}_u\right)\mathfrak{F}_v\cdot (f)\mathfrak{F}_{u\ast v}=\left((f)\mathfrak{F}_u\right)\mathfrak{F}_v, \qquad  \hbox{для довільного ідемпотента} \quad f\in S;
\end{equation*}
   \item[$(vi)$] якщо $u,v\in T_S$, то
\begin{equation*}
    f\cdot(g)\mathfrak{F}_u\leqslant e_{u\ast v},
\end{equation*}
   для всіх ідемпотентів $f\leqslant e_u$ та $g\leqslant e_v$ напівгрупи $S$.
\end{itemize}
\end{lemma}

\begin{theorem}[{\cite[теорема~3]{McFadden-Carroll-1971}}]\label{theorem-3.3}
Нехай $S$~--- $F$-інверсна напівгрупа і $\mathscr{S}=\bigcup_{t\in T_S}\left({\downarrow}e_t\times\{t\}\right)$. Означимо на $\mathscr{S}$ бінарну операцію $\circ$ наступним чином: якщо $u,v\in T_S$, то для ідем\-по\-тен\-тів $f\leqslant e_u$ та $g\leqslant e_v$ покладемо
\begin{equation}\label{eq-circ}
    (f,u)\circ(g,v)=\left(f\cdot(g)\mathfrak{F}_u,u\ast v\right).
\end{equation}
Тоді $\circ$~--- напівгрупова операція на $\mathscr{S}$ і напівгрупа $\left(\mathscr{S},\circ\right)$ ізоморфна напівгрупі $S$ стосовно відображення $\mathfrak{H}\colon S\rightarrow \mathscr{S}\colon s\mapsto\left(ss^{-1},t_s\right)$.
\end{theorem}

Нехай $A$ та $B$~--- напівгрупи, $\operatorname{\textsf{End}}(B)$~--- напівгрупа ендоморфізмів напівгрупи $B$ та визначено гомоморфізм $\mathfrak{h}\colon A\rightarrow \operatorname{\textsf{End}}(B)\colon b\mapsto \mathfrak{h}_b$. Тоді множина $A\times B$ з бінарною операцією
\begin{equation*}
    (a_1,b_1)\cdot(a_2,b_2)=\left(a_1a_2,(b_1)\mathfrak{h}_{a_2}b_2\right), \qquad a_1,a_2\in A, \; b_1,b_2\in B,
\end{equation*}
називається \emph{напівпрямим добутком} напівгрупи $A$ напівгрупою $B$ стосовно гомо\-мор\-фіз\-му $\mathfrak{h}$ і позначається $A\ltimes_\mathfrak{h}B$ \cite{Lawson1998}. У цьому випадку кажуть, що визначена права дія напівгрупи $A$ на напівгрупі $B$ ендоморфізмами (гомоморфізмами). Зауважимо, що напівпрямий добуток інверсних напівгруп не завжди є інверсною напівгрупою (див. \cite[розділ~5.3]{Lawson1998}).

\begin{lemma}\label{lemma-3.4}
Відображення $\mathfrak{h}\colon H(1)\rightarrow\operatorname{\textsf{End}}\left(E(\mathscr{P\!\!H}^+_{\!\!\operatorname{\textsf{cf}}}\!(\mathbb{R}))\right) \colon \gamma\mapsto\mathfrak{h}_\gamma$, де $(\alpha)\mathfrak{h}_\gamma=\gamma^{-1}\alpha\gamma$~--- ав\-то\-мор\-фіз\-м напівгратки $E(\mathscr{P\!\!H}^+_{\!\!\operatorname{\textsf{cf}}}\!(\mathbb{R}))$, є гомоморфізмом, причому $\mathfrak{h}_1$~--- тотожний ав\-то\-мор\-фізм на\-пів\-грат\-ки $E(\mathscr{P\!\!H}^+_{\!\!\operatorname{\textsf{cf}}}\!(\mathbb{R}))$.
\end{lemma}

\begin{proof}
Для довільних $\gamma\in H(1)$, $\varepsilon,\iota\in E(\mathscr{P\!\!H}^+_{\!\!\operatorname{\textsf{cf}}}\!(\mathbb{R}))$ маємо, що
\begin{equation*}
    (\varepsilon\iota)\mathfrak{h}_\gamma=\gamma^{-1}\varepsilon\iota\gamma=\gamma^{-1}\varepsilon\gamma\gamma^{-1}\iota\gamma= (\varepsilon)\mathfrak{h}_\gamma(\iota)\mathfrak{h}_\gamma,
\end{equation*}
а отже $\mathfrak{h}_\gamma$~--- гомоморфізм напівгратки $E(\mathscr{P\!\!H}^+_{\!\!\operatorname{\textsf{cf}}}\!(\mathbb{R}))$. Також, оскільки для довільних $\gamma\in H(1)$ та $\varepsilon\in E(\mathscr{P\!\!H}^+_{\!\!\operatorname{\textsf{cf}}}\!(\mathbb{R}))$ елемент $\gamma\varepsilon\gamma^{-1}$ є ідемпотентом напівгрупи $\mathscr{P\!\!H}^+_{\!\!\operatorname{\textsf{cf}}}\!(\mathbb{R})$ і $(\gamma\varepsilon\gamma^{-1})\mathfrak{h}_\gamma=\gamma^{-1}(\gamma\varepsilon\gamma^{-1})\gamma=\varepsilon$, то гомоморфізм $\mathfrak{h}_\gamma$ є сюр'єктивним відображенням. Очевидно, що $\mathfrak{h}_1$~--- тотожне відображення напівгратки $E(\mathscr{P\!\!H}^+_{\!\!\operatorname{\textsf{cf}}}\!(\mathbb{R}))$.

Припустимо, що $(\varepsilon)\mathfrak{h}_\gamma=(\iota)\mathfrak{h}_\gamma$, для деяких $\gamma\in H(1)$, $\varepsilon,\iota\in E(\mathscr{P\!\!H}^+_{\!\!\operatorname{\textsf{cf}}}\!(\mathbb{R}))$. Оскільки $H(1)$~--- група одиниць напівгрупи $\mathscr{P\!\!H}^+_{\!\!\operatorname{\textsf{cf}}}\!(\mathbb{R})$, то з рівностей
\begin{equation*}
    \gamma^{-1}\varepsilon\gamma=(\varepsilon)\mathfrak{h}_\gamma=(\iota)\mathfrak{h}_\gamma=\gamma^{-1}\iota\gamma
\end{equation*}
випливає, що
\begin{equation*}
    \varepsilon=1\varepsilon 1=\gamma\gamma^{-1}\varepsilon\gamma\gamma^{-1}=\gamma\gamma^{-1}\iota\gamma\gamma^{-1}=1\iota 1=\iota,
\end{equation*}
а отже $\mathfrak{h}_\gamma$~--- ав\-то\-мор\-фіз\-м напівгратки $E(\mathscr{P\!\!H}^+_{\!\!\operatorname{\textsf{cf}}}\!(\mathbb{R}))$.

Зафіксуємо довільні $\gamma,\delta\in H(1)$. Тоді для довільного ідемпотента $\varepsilon\in \mathscr{P\!\!H}^+_{\!\!\operatorname{\textsf{cf}}}\!(\mathbb{R})$ маємо, що
\begin{equation*}
    (\varepsilon)\mathfrak{h}_{\gamma\delta}= (\gamma\delta)^{-1}\varepsilon\gamma\delta= \delta^{-1}\gamma^{-1}\varepsilon\gamma\delta= \delta^{-1}(\varepsilon)\mathfrak{h}_{\gamma}\delta= \left((\varepsilon)\mathfrak{h}_{\gamma}\right)\mathfrak{h}_\delta= (\varepsilon)\left(\mathfrak{h}_{\gamma}\cdot\mathfrak{h}_\delta\right),
\end{equation*}
а отже так означене відображення $\mathfrak{h}\colon H(1)\rightarrow\operatorname{\textsf{End}}\left(E(\mathscr{P\!\!H}^+_{\!\!\operatorname{\textsf{cf}}}\!(\mathbb{R}))\right)$ є гомоморфізмом.
\end{proof}

Наступна теорема описує структуру напівгрупи $\mathscr{P\!\!H}^+_{\!\!\operatorname{\textsf{cf}}}\!(\mathbb{R})$.

\begin{theorem}\label{theorem-3.5}
Напівгрупа $\mathscr{P\!\!H}^+_{\!\!\operatorname{\textsf{cf}}}\!(\mathbb{R})$ ізоморфна напівпрямому добутку $\mathscr{H}^+\!(\mathbb{R})\ltimes_\mathfrak{h}\mathscr{P}_{\!\infty}(\mathbb{R})$ вільної напівгратки з одиницею $(\mathscr{P}_{\!\infty}(\mathbb{R}),\cup)$ групою $\mathscr{H}^+\!(\mathbb{R})$ усіх гомеоморфізмів, що зберігають орієнтацію прос\-то\-ру $\mathbb{R}$.
\end{theorem}

\begin{proof}
Оскільки група одиниць $H(1)$ напівгрупи $\mathscr{P\!\!H}^+_{\!\!\operatorname{\textsf{cf}}}\!(\mathbb{R})$ ізоморфна групі $\mathscr{H}^+\!(\mathbb{R})$ усіх гомеоморфізмів, що зберігають орієнтацію прос\-то\-ру $\mathbb{R}$, то за твердженням~\ref{proposition-2.1}$(ii)$ нам достатньо показати, що напівгрупа $\mathscr{P\!\!H}^+_{\!\!\operatorname{\textsf{cf}}}\!(\mathbb{R})$ ізоморфна напівпрямому добутку $H(1)\ltimes_\mathfrak{h}E(\mathscr{P\!\!H}^+_{\!\!\operatorname{\textsf{cf}}}\!(\mathbb{R}))$ на\-пів\-грат\-ки $E(\mathscr{P\!\!H}^+_{\!\!\operatorname{\textsf{cf}}}\!(\mathbb{R}))$ групою одиниць $H(1)$ напівгрупи $\mathscr{P\!\!H}^+_{\!\!\operatorname{\textsf{cf}}}\!(\mathbb{R})$ сто\-сов\-но гомоморфізму $\mathfrak{h}\colon H(1)\rightarrow\operatorname{\textsf{End}}\left(E(\mathscr{P\!\!H}^+_{\!\!\operatorname{\textsf{cf}}}\!(\mathbb{R}))\right) \colon \gamma\mapsto\mathfrak{h}_\gamma$, де $(\alpha)\mathfrak{h}_\gamma=\gamma^{-1}\alpha\gamma$.

Означимо відображення $\mathfrak{T}\colon\mathscr{P\!\!H}^+_{\!\!\operatorname{\textsf{cf}}}\!(\mathbb{R})\rightarrow H(1)\ltimes_\mathfrak{h}E(\mathscr{P\!\!H}^+_{\!\!\operatorname{\textsf{cf}}}\!(\mathbb{R}))$ за формулою
\begin{equation*}
    (\alpha)\mathfrak{T}=\left(\gamma_\alpha,\alpha^{-1}\alpha\right),
\end{equation*}
де елемент $\gamma_\alpha$ групи одиниць $H(1)$ напівгрупи $\mathscr{P\!\!H}^+_{\!\!\operatorname{\textsf{cf}}}\!(\mathbb{R})$, визначений у твердженні~\ref{propositiony-2.12}. З тверд\-жен\-ня~\ref{proposition-2.1} і наслідку~\ref{corollary-3.1} випливає, що відображення $\mathfrak{T}\colon\mathscr{P\!\!H}^+_{\!\!\operatorname{\textsf{cf}}}\!(\mathbb{R})\rightarrow H(1)\ltimes_\mathfrak{h}E(\mathscr{P\!\!H}^+_{\!\!\operatorname{\textsf{cf}}}\!(\mathbb{R}))$ означене коректно, і воно є сюр'єктивним. Припустимо, що існують еле\-мен\-ти $\alpha$ та $\beta$ напівгрупи $\mathscr{P\!\!H}^+_{\!\!\operatorname{\textsf{cf}}}\!(\mathbb{R})$ такі, що $(\alpha)\mathfrak{T}=(\beta)\mathfrak{T}$. Тоді $\left(\gamma_\alpha,\alpha^{-1}\alpha\right)=\left(\gamma_\beta,\beta^{-1}\beta\right)$ і, використавши твердження~\ref{propositiony-2.12} i лему~1.4.6 з \cite{Lawson1998}, отримуємо
\begin{equation*}
    \alpha=\gamma_\alpha\alpha^{-1}\alpha=\gamma_\beta\beta^{-1}\beta=\beta,
\end{equation*}
а отже відображення $\mathfrak{T}\colon\mathscr{P\!\!H}^+_{\!\!\operatorname{\textsf{cf}}}\!(\mathbb{R})\rightarrow H(1)\ltimes_\mathfrak{h}E(\mathscr{P\!\!H}^+_{\!\!\operatorname{\textsf{cf}}}\!(\mathbb{R}))$ є сюр'єктивним.

Нехай $\alpha$ та $\beta$ --- довільні елементи напівгрупи $\mathscr{P\!\!H}^+_{\!\!\operatorname{\textsf{cf}}}\!(\mathbb{R})$. Тоді за  теоремою~\ref{theorem-2.13} має\-мо, що $\alpha\beta\mathfrak{C}_{\textsf{mg}}\gamma_\alpha\gamma_\beta$, і оскільки $\gamma_\alpha,\gamma_\beta\in H(1)$, то отримуємо, що $\gamma_\alpha\gamma_\beta=\gamma_{\alpha\beta}$. Звідси вип\-ли\-ває, що
\begin{equation*}
    (\alpha)\mathfrak{T}(\beta)\mathfrak{T}= \left(\gamma_\alpha,\alpha^{-1}\alpha\right)\left(\gamma_\beta,\beta^{-1}\beta\right)= \left(\gamma_\alpha\gamma_\beta,\gamma_\beta^{-1}\alpha^{-1}\alpha\gamma_\beta\beta^{-1}\beta\right)= \left(\gamma_{\alpha\beta},\gamma_\beta^{-1}\alpha^{-1}\alpha\gamma_\beta\beta^{-1}\beta\right).
\end{equation*}
За лемою~\ref{lemma-3.4} відображення $\mathfrak{h}_\gamma\colon E(\mathscr{P\!\!H}^+_{\!\!\operatorname{\textsf{cf}}}\!(\mathbb{R}))\rightarrow E(\mathscr{P\!\!H}^+_{\!\!\operatorname{\textsf{cf}}}\!(\mathbb{R}))\colon\alpha\mapsto\gamma^{-1}\alpha\gamma$ є ав\-то\-мор\-фіз\-мом напівгратки $E(\mathscr{P\!\!H}^+_{\!\!\operatorname{\textsf{cf}}}\!(\mathbb{R}))$, а отже отримуємо, що елемент $\gamma_\beta^{-1}\alpha^{-1}\alpha\gamma_\beta$ є ідем\-по\-тен\-том напівгрупи $\mathscr{P\!\!H}^+_{\!\!\operatorname{\textsf{cf}}}\!(\mathbb{R})$. Оскільки $\mathscr{P\!\!H}^+_{\!\!\operatorname{\textsf{cf}}}\!(\mathbb{R})$~--- інверсна напівгрупа, то за твердженням~\ref{propositiony-2.12} i лемою~1.4.6 з \cite{Lawson1998} маємо, що  $\beta=\gamma_\beta\beta^{-1}\beta$, а отже:
\begin{equation*}
\begin{split}
  \gamma_\beta^{-1}\alpha^{-1}\alpha\gamma_\beta\beta^{-1}\beta & = \left(\gamma_\beta^{-1}\alpha^{-1}\alpha\gamma_\beta\right)\left(\beta^{-1}\beta\right)\left(\beta^{-1}\beta\right)=\\
    & = \left(\beta^{-1}\beta\gamma_\beta^{-1}\right)\left(\alpha^{-1}\alpha\right)\left(\gamma_\beta\beta^{-1}\beta\right)=\\
    & = \left(\gamma_\beta\beta^{-1}\beta\right)^{-1}\left(\alpha^{-1}\alpha\right)\left(\gamma_\beta\beta^{-1}\beta\right)=\\
    & = \beta^{-1}\left(\alpha^{-1}\alpha\right)\beta=\\
    & = \left(\beta^{-1}\alpha^{-1}\right)\left(\alpha\beta\right)=\\
    & = \left(\alpha\beta\right)^{-1}\left(\alpha\beta\right).
\end{split}
\end{equation*}
Таким чином, отримуємо
\begin{equation*}
    (\alpha\beta)\mathfrak{T}=\left(\gamma_{\alpha\beta},(\alpha\beta)^{-1}\alpha\beta\right)=(\alpha)\mathfrak{T}(\beta)\mathfrak{T},
\end{equation*}
а отже відображення $\mathfrak{T}\colon\mathscr{P\!\!H}^+_{\!\!\operatorname{\textsf{cf}}}\!(\mathbb{R})\rightarrow H(1)\ltimes_\mathfrak{h}E(\mathscr{P\!\!H}^+_{\!\!\operatorname{\textsf{cf}}}\!(\mathbb{R}))$ є гомоморфізмом, що і за\-вер\-шує доведення нашої теореми.
\end{proof}

\section*{Подяка}

Автори виражають щиру подяку рецензенту за корисні зауваження та коментарі до рукопису цієї праці.


\end{document}